\documentclass[reqno]{amsart}

\usepackage[latin1]{inputenc}
\usepackage{amssymb}
\usepackage{graphicx}
\usepackage{amscd}
\usepackage[hidelinks]{hyperref}
\usepackage{color}
\usepackage{float}
\usepackage{graphics,amsmath,amssymb}
\usepackage{amsthm}
\usepackage{amsfonts}
\usepackage{latexsym}
\usepackage{epsf}
\usepackage{enumerate}
\usepackage{xifthen}
\usepackage{mathrsfs}
\usepackage{dsfont}
\usepackage{makecell}
\usepackage[FIGTOPCAP]{subfigure}
\usepackage{amsmath}
\allowdisplaybreaks[4]
\usepackage{listings}
\usepackage{etoolbox}
\usepackage{fancyhdr}

 \newtheoremstyle{mytheorem}
 {3pt}
 {3pt}
 {\slshape}
 {}
 {\bfseries}
 {.}
 { }
 {}

\numberwithin{equation}{section}

\theoremstyle{theorem}
\newtheorem{theorem}{Theorem}[section]

\newtheorem{lemma}[theorem]{Lemma}

\theoremstyle{definition}

\newtheorem{remark}{Remark}[section]

\newcommand{\Keywords}[1]{\ifthenelse{\isempty{#1}}{}{\smallskip \smallskip \noindent \textbf{Keywords}. #1}}
\newcommand{\MSC}[2][2010]{\ifthenelse{\isempty{#2}}{}{\smallskip \smallskip \noindent \textbf{#1MSC}. #2}}
\newcommand{\abstractnote}[1]{\ifthenelse{\isempty{#1}}{}{\smallskip \smallskip \noindent \textsuperscript{\dag}#1}}

\makeatletter
\def\specialsection{\@startsection{section}{1}%
  \z@{\linespacing\@plus\linespacing}{.5\linespacing}%
  {\normalfont}}
\def\section{\@startsection{section}{1}%
  \z@{.7\linespacing\@plus\linespacing}{.5\linespacing}%
  {\normalfont\scshape}}
\patchcmd{\@settitle}{\uppercasenonmath\@title}{\Large\boldmath}{}{}
\patchcmd{\@settitle}{\begin{center}}{\begin{flushleft}}{}{}
\patchcmd{\@settitle}{\end{center}}{\end{flushleft}}{}{}
\patchcmd{\@setauthors}{\MakeUppercase}{\normalsize}{}{}
\patchcmd{\@setauthors}{\centering}{\raggedright}{}{}
\patchcmd{\section}{\scshape}{\large\bfseries\boldmath}{}{}
\patchcmd{\subsection}{\bfseries}{\bfseries\boldmath}{}{}
\renewcommand{\@secnumfont}{\bfseries}
\patchcmd{\@startsection}{\@afterindenttrue}{\@afterindentfalse}{}{}
\patchcmd{\abstract}{\leftmargin3pc}{\leftmargin1pc}{}{}

\def\maketitle{\par
  \@topnum\z@ 
  \@setcopyright
  \thispagestyle{empty}
  \ifx\@empty\shortauthors \let\shortauthors\shorttitle
  \else \andify\shortauthors
  \fi
  \@maketitle@hook
  \begingroup
  \@maketitle
  \toks@\@xp{\shortauthors}\@temptokena\@xp{\shorttitle}%
  \toks4{\def\\{ \ignorespaces}}
  \edef\@tempa{%
    \@nx\markboth{\the\toks4
      \@nx\MakeUppercase{\the\toks@}}{\the\@temptokena}}%
  \@tempa
  \endgroup
  \c@footnote\z@
  \@cleartopmattertags
}
\makeatother

\title[Arithmetic properties for cubic partition pairs]{Arithmetic properties for cubic partition pairs modulo powers of 3}

\author[S. Chern]{Shane Chern}
\address{Department of Mathematics, Pennsylvania State University, University Park, PA 16802, USA}
\email{shanechern@psu.edu; chenxiaohang92@gmail.com}

\date{}

\begin{document}

%

\maketitle

\begin{abstract}

Let $b(n)$ denote the number of cubic partition pairs of $n$. In this paper, we aim to provide a strategy to obtain arithmetic properties of $b(n)$. This gives affirmative answers to two of Lin's conjectures.


\MSC{Primary 05A17; Secondary 11P83.}
\end{abstract}

\section{Introduction and main results}

A partition of a natural number $n$ is a nonincreasing sequence of positive integers whose sum equals $n$. Let $p(n)$ be the number of partitions. One of Ramanujan's most beautiful work is the following three congruences
\begin{align*}
p(5n+4)&\equiv 0 \pmod{5},\\
p(7n+5)&\equiv 0 \pmod{7},\\
p(11n+6)&\equiv 0 \pmod{11}.
\end{align*}

Motivated by Ramanujan's result, Chan \cite{Ch2010} introduced the cubic partition function $a(n)$ with generating function given by
\begin{equation}\label{eq:a}
\sum_{n\ge 0}a(n)q^n = \frac{1}{(q;q)_\infty (q^2;q^2)_\infty},\quad |q|<1,
\end{equation}
where and in the sequel, we use the standard notation
$$(a;q)_\infty = \prod_{n\ge 0}(1-aq^n).$$
One of the analogous congruences to Ramanujan's result is
\begin{equation}
a(3n+2)\equiv 0 \pmod{3},
\end{equation}
which can be deduced easily from the following identity obtained by Chan
\begin{equation}
\sum_{n\ge 0} a(3n+2)q^n=3\frac{(q^3;q^3)_\infty^3(q^6;q^6)_\infty^3}{(q;q)_\infty^4(q^2;q^2)_\infty^4}.
\end{equation}
Later on, many authors studied other Ramanujan-like congruences for $a(n)$. For instance, Chen and Lin \cite{CL2009} found four new congruences modulo $7$, while recently the author and Dastidar \cite{CD2016} obtained two congruences modulo $11$.

After Chan's work, many authors also investigated analogous partition functions. For instance, in 2011, Zhao and Zhong \cite{ZZ2011} studied the cubic partition pair function $b(n)$ given by
\begin{equation}\label{eq:b}
\sum_{n\ge 0}b(n)q^n = \frac{1}{(q;q)_\infty^2 (q^2;q^2)_\infty^2}.
\end{equation}
They obtained several congruences similar to Ramanujan's result:
\begin{align*}
b(5n+4)&\equiv 0 \pmod{5},\\
b(7n+i)&\equiv 0 \pmod{7},\\
b(9n+7)&\equiv 0 \pmod{9},
\end{align*}
where $i=2,3,4,6$. For combinatorial proofs of the first two congruences, the reader may refer to Kim \cite{Ki2011} and Zhou \cite{Zh2012} respectively. In a recent paper, Lin \cite{Li2017} studied the arithmetic properties of $b(n)$ modulo $27$. For instance, he showed the following infinite families of congruences
$$b\left(81\left(7^{2\alpha+1}(7n+k)+\frac{7^{2\alpha+2}-1}{12}\right)+7\right)\equiv 0 \pmod{27}$$
for $\alpha\ge 0$ and $k=1,2,3,4,5,6$. He also conjectured the following three congruences modulo powers of $3$.
\begin{equation}\label{eq:lin1}
b(81n+61)\equiv 0\pmod{243},
\end{equation}
and
\begin{align}
\sum_{n\ge 0} b(81n+7)q^n&\equiv 9\frac{(q^2;q^2)_\infty(q^3;q^3)_\infty^2}{(q^6;q^6)_\infty}\pmod{81},\label{eq:lin2}\\
\sum_{n\ge 0} b(81n+34)q^n&\equiv 36\frac{(q;q)_\infty(q^6;q^6)_\infty^2}{(q^3;q^3)_\infty}\pmod{81}.\label{eq:lin3}
\end{align}

In this paper, we will provide a strategy to give affirmative answers to Lin's conjectures. In fact, for \eqref{eq:lin1}, we have a stronger result
\begin{theorem}\label{th:1.1}
For any $n\ge 0$,
\begin{equation}\label{eq:th1}
b(81n+61)\equiv 0\pmod{729}.
\end{equation}
\end{theorem}
\noindent We also obtain
\begin{theorem}\label{th:1.1.5}
For any $n\ge 0$,
\begin{equation}\label{eq:th1.5}
b(243n+61)\equiv 0 \pmod{2187}.
\end{equation}
\end{theorem}
\noindent For the remaining two congruences, we show
\begin{theorem}\label{th:1.2}
Eqs. \eqref{eq:lin2} and \eqref{eq:lin3} hold for any $n\ge 0$.
\end{theorem}

\section{Preliminary results}

We first introduce two Ramanujan's theta functions $\varphi(q)$ and $\psi(q)$. They are defined by
\begin{align*}
\varphi(q)&=\frac{(q^2;q^2)_\infty^5}{(q;q)_\infty^2(q^4;q^4)_\infty^2}=\sum_{n=-\infty}^\infty q^{n^2},\\
\psi(q)&=\frac{(q^2;q^2)_\infty^2}{(q;q)_\infty}=\sum_{n\ge 0}q^{n(n+1)/2}.
\end{align*}

Let
$$w(q)=\frac{(q;q)_\infty(q^6;q^6)_\infty^3}{(q^2;q^2)_\infty(q^3;q^3)_\infty^3},$$
In two recent papers, Zhao and Zhong \cite{ZZ2011}, and Baruah and Ojah \cite{BO2011} obtained the 3-dissection of $\varphi(-q)$ and $1/\varphi(-q)$ respectively. Their results are

\begin{lemma}\label{le:2.2}
It holds that
\begin{equation}\label{eq:2.3}
\varphi(-q)=\varphi(-q^9)(1-2qw(q^3)).
\end{equation}
\end{lemma}

\begin{lemma}\label{le:2.3}
It holds that
\begin{equation}\label{eq:2.4}
\frac{1}{\varphi(-q)}=\frac{\varphi(-q^9)^3}{\varphi(-q^3)^4}(1+2qw(q^3)+(2qw(q^3))^2).
\end{equation}
\end{lemma}

\noindent For convenience, we write $\xi(q):=2qw(q^3)$. Multiplying \eqref{eq:2.3} and \eqref{eq:2.4}, we have
$$\frac{\varphi(-q^9)^4}{\varphi(-q^3)^4}(1-(2qw(q^3))^3)=1.$$
It follows that
\begin{equation}\label{eq:2.5}
\xi(q)^3=1-\frac{\varphi(-q^3)^4}{\varphi(-q^9)^4}.
\end{equation}

We also notice that in \cite{Sh1994}, Shen introduced 24 identities involving Ramanujan's theta functions. His Eq. (3.2) states

\begin{lemma}\label{le:2.1}
It holds that
\begin{equation}\label{eq:2.1}
\frac{\psi(q^3)^3}{\psi(q)}=\frac{1}{8q}\left(\frac{\varphi(-q^3)^3}{\varphi(-q)}-\frac{\varphi(-q)^3}{\varphi(-q^3)}\right).
\end{equation}
\end{lemma}

\noindent By Lemmas \ref{le:2.2} and \ref{le:2.3}, we can rewrite it in the following form which will be frequently used in our proof:
\begin{align}\label{eq:2.2}
\frac{1}{\psi(q)}&=\frac{1}{8q\psi(q^3)^3}\left(\frac{\varphi(-q^3)^3}{\varphi(-q)}-\frac{\varphi(-q)^3}{\varphi(-q^3)}\right)\notag \\
&= \frac{1}{8q\psi(q^3)^3}\left(\frac{\varphi(-q^9)^3}{\varphi(-q^3)}(1+\xi(q)+\xi(q)^2)-\frac{\varphi(-q^9)^3}{\varphi(-q^3)}(1-\xi(q))^3\right)\notag \\
&= \frac{1}{8q\psi(q^3)^3}\frac{\varphi(-q^9)^3}{\varphi(-q^3)}(4\xi(q)-2\xi(q)^2+\xi(q)^3).
\end{align}

Given a series $\sum_{n\ge 0}u(n)q^n$, we say it is \textit{$(s,t,\lambda;4)$-expressible} if we can find $\alpha\in\mathbb{Z}$, $\beta\in\mathbb{Z}_{>0}$ with $\beta \equiv \lambda\pmod{4}$, and three finite sequences $(s_1,s_2,\dots,s_n)\in\mathbb{Z}^n$, $(t_1,t_2,\dots,t_n)\in\mathbb{Z}^n$, $(v_1,v_2,\dots,v_n):=(v(s_1,t_1),v(s_2,t_2),\ldots,v(s_n,t_n))\in\mathbb{Q}^n$ with $s_1\equiv s_2\equiv\cdots\equiv s_n\equiv s \pmod{4}$, $t_1\equiv t_2\equiv\cdots\equiv t_n\equiv t \pmod{4}$, and $s_1+t_1=s_2+t_2=\cdots=s_n+t_n$, such that
$$\sum_{n\ge 0}u(n)q^n=\frac{q^\alpha}{\psi(q)^\beta}\sum_{i=1}^n v(s_i,t_i)\varphi(-q)^{s_i}\varphi(-q^3)^{t_i}.$$
The 3-dissection of a \textit{$(s,t,\lambda;4)$-expressible} series has the following property:

\begin{lemma}\label{le:2.4}
If $s$, $t$, and $\lambda$ satisfy one of the following conditions:
\begin{enumerate}[{\rm\indent(1)}]
\item $\lambda=0$, and $s\equiv t \pmod{4}$;
\item $\lambda=2$, and $s\equiv t+2 \pmod{4}$,
\end{enumerate}
then given an $(s,t,\lambda;4)$-expressible series $\sum_{n\ge 0}u(n)q^n$, we can find an integer $\ell\in\{0,1,2\}$ such that $\sum_{n\ge 0}u(3n+\ell)q^n$ is also $(s,t,\lambda;4)$-expressible.
\end{lemma}

\begin{proof}
If $\sum_{n\ge 0}u(n)q^n$ is $(a,b,\lambda;4)$-expressible, we can write it as
$$\sum_{n\ge 0}u(n)q^n=\frac{q^\alpha}{\psi(q)^\beta}\sum_{i=1}^n v(s_i,t_i)\varphi(-q)^{s_i}\varphi(-q^3)^{t_i}.$$
By \eqref{eq:2.2}, we have
$$\sum_{n\ge 0}u(n)q^n=\frac{q^{\alpha-\beta}}{8^\beta\psi(q^3)^{3\beta}}\sum_{n\ge 0}g(n)q^n,$$
where
$$\sum_{n\ge 0}g(n)q^n=\left(\frac{\varphi(-q^9)^3}{\varphi(-q^3)}\right)^\beta(4\xi(q)-2\xi(q)^2+\xi(q)^3)^\beta\sum_{i=1}^n v(s_i,t_i)\varphi(-q)^{s_i}\varphi(-q^3)^{t_i}.$$
Suppose that $\alpha-\beta\equiv \ell\pmod{3}$, we have
$$\sum_{n\ge 0}u(3n+\ell)q^n=\frac{q^{(\alpha-\beta-\ell)/3}}{8^\beta\psi(q)^{3\beta}}\sum_{n\ge 0}g(3n)q^n.$$
Here note that if $\beta\equiv\lambda=0$ or $2\pmod{4}$, we have $3\beta\equiv\lambda\pmod{4}$.

Next for any $i=1,\ldots,n$, we aim to obtain the 3-dissection of
$$\sum_{n\ge 0}g_i(n)q^n=\varphi(-q)^{s_i}\varphi(-q^3)^{t_i}\left(\frac{\varphi(-q^9)^3}{\varphi(-q^3)}\right)^\beta(4\xi(q)-2\xi(q)^2+\xi(q)^3)^\beta.$$
If $s_i\ge 0$, then by Lemma \ref{le:2.2},
\begin{align*}
&\varphi(-q)^{s_i}\varphi(-q^3)^{t_i}\left(\frac{\varphi(-q^9)^3}{\varphi(-q^3)}\right)^\beta(4\xi(q)-2\xi(q)^2+\xi(q)^3)^\beta\\
&\quad=\varphi(-q^3)^{t_i-\beta}\varphi(-q^9)^{s_i+3\beta}(1-\xi(q))^{s_i}(4\xi(q)-2\xi(q)^2+\xi(q)^3)^\beta.
\end{align*}
Now we pick up terms of the form $q^{3n}$ and obtain
$$\varphi(-q^3)^{t_i-\beta}\varphi(-q^9)^{s_i+3\beta}\sum_{j=0}^{\left\lfloor\frac{s_i+3\beta}{3}\right\rfloor}m_{3j}\xi(q)^{3j}$$
for some $m_{3j}$. By \eqref{eq:2.5}, we therefore have
\begin{align*}
\sum_{n\ge 0}g_i(3n)q^n&=\varphi(-q)^{t_i-\beta}\varphi(-q^3)^{s_i+3\beta}\sum_{j=0}^{\left\lfloor\frac{s_i+3\beta}{3}\right\rfloor}m_{3j}\left(1-\frac{\varphi(-q)^4}{\varphi(-q^3)^4}\right)^j\\
&=\sum_{j=0}^{\left\lfloor\frac{s_i+3\beta}{3}\right\rfloor}v'(s_{i,j}',t_{i,j}')\varphi(-q)^{s_{i,j}'}\varphi(-q^3)^{t_{i,j}'},
\end{align*}
where
$$\begin{cases}
s_{i,j}'=t_i-\beta+4j,\\
t_{i,j}'=s_i+3\beta-4j,
\end{cases}$$
for $j=0,\ldots,\lfloor(s_i+3\beta)/3\rfloor$. For both conditions (i) and (ii), it is easy to see that $s_{i,j}'\equiv s_i\pmod{4}$ and $t_{i,j}'\equiv t_i\pmod{4}$ for all $j$. Similarly, if $s_j<0$, we can write $\sum_{n\ge 0}g_i(3n)q^n$ as a linear combination of $\varphi(-q)^{s_{i,j}'}\varphi(-q^3)^{t_{i,j}'}$ where
$$\begin{cases}
s_{i,j}'=t_i+4s_i-\beta+4j,\\
t_{i,j}'=-3s_i+3\beta-4j,
\end{cases}$$
for $j=0,\ldots,\lfloor(-2s_i+3\beta)/3\rfloor$. For both conditions (i) and (ii), we also see that $s_{i,j}'\equiv s_i\pmod{4}$ and $t_{i,j}'\equiv t_i\pmod{4}$ for all $j$. Furthermore, we have $s_{i,j}'+t_{i,j}'=s_i+t_i+2\beta$ for all $j$ whenever $s_i\ge 0$ or $s_i<0$.

Since the above argument holds for all $i=1,\ldots,n$, we therefore can find three new finite sequences $(s_1',s_2',\dots,s_{n'}')\in\mathbb{Z}^{n'}$, $(t_1',t_2',\dots,t_{n'}')\in\mathbb{Z}^{n'}$, $(v_1',v_2',\dots,v_{n'}'):=(v'(s_1',t_1'),v'(s_2',t_2'),\ldots,v'(s_{n'}',t_{n'}'))\in\mathbb{Q}^{n'}$ with $s_1'\equiv s_2'\equiv\cdots\equiv s_{n'}'\equiv s \pmod{4}$, $t_1'\equiv t_2'\equiv\cdots\equiv t_{n'}'\equiv t \pmod{4}$, and $s_1'+t_1'=s_2'+t_2'=\cdots=s_{n'}'+t_{n'}'$, such that
$$\sum_{n\ge 0}u(3n+\ell)q^n=\frac{q^{(\alpha-\beta-\ell)/3}}{\psi(q)^{3\beta}}\sum_{i=1}^{n'} v'(s_i',t_i')\varphi(-q)^{s_i'}\varphi(-q^3)^{t_i'}.$$
\end{proof}

\begin{remark}
The aim of this lemma is to show that if a series $\sum_{n\ge 0}u(n)q^n$ is $(s,t,\lambda;4)$-expressible satisfying conditions (i) or (ii), then we may write
$$\sum_{n\ge 0}u(3^kn+\ell)q^n$$
(for some computable $\ell$) as the form
$$\frac{q^\alpha}{\psi(q)^\beta}\sum_{i=1}^n v(s_i,t_i)\varphi(-q)^{s_i}\varphi(-q^3)^{t_i}.$$
In this sense, we can study the arithmetic property of $u(3^kn+\ell)$ by studying the coefficients $v(s_i,t_i)$.
\end{remark}

\section{Proofs}

In view of \eqref{eq:b}, we can rewrite the generating function of $b(n)$ as
\begin{equation}\label{eq:3.1}
\sum_{n\ge 0}b(n)q^n=\frac{1}{\psi(q)^2\varphi(-q)^2}=\frac{\varphi(-q^3)^0}{\psi(q)^2\varphi(-q)^2}.
\end{equation}
Hence $\sum_{n\ge 0}b(n)q^n$ is $(2,0,2;4)$-expressible. According to Lemma \ref{le:2.4}, we can see that for any positive integer $\alpha$, we can find some $\ell$, such that $\sum_{n\ge 0}b(3^\alpha n+\ell)q^n$ has the form
$$\frac{q^\alpha}{\psi(q)^\beta}\sum_{i=1}^n v(s_i,t_i)\varphi(-q)^{s_i}\varphi(-q^3)^{t_i},$$
where $s_1\equiv s_2\equiv\cdots\equiv s_n\equiv 2 \pmod{4}$, $t_1\equiv t_2\equiv\cdots\equiv t_n\equiv 0 \pmod{4}$, and $s_1+t_1=s_2+t_2=\cdots=s_n+t_n$.

Using the method presented in the proof of Lemma \ref{le:2.4}, we first give the detail of the 3-dissection of $\sum_{n\ge 0}b(n)q^n$.
\begin{align*}
\frac{1}{\psi(q)^2\varphi(-q)^2}&=\frac{q^{-2}}{2^6\psi(q^3)^6}\left(\frac{\varphi(-q^9)^3}{\varphi(-q^3)}\right)^2(4\xi(q)-2\xi(q)^2+\xi(q)^3)^2\\
&\quad\quad\times\left(\frac{\varphi(-q^9)^3}{\varphi(-q^3)^4}\right)^2(1+\xi(q)+\xi(q)^2)^2\\
&=\frac{q^{-3}\cdot q}{2^6\psi(q^3)^6}\frac{\varphi(-q^9)^{12}}{\varphi(-q^3)^{10}}(4\xi(q)-2\xi(q)^2+\xi(q)^3)^2(1+\xi(q)+\xi(q)^2)^2.
\end{align*}

Extracting terms involving $q^{3n+1}$ and replacing $q^{3n}$ by $q^n$, we have
$$\sum_{n\ge 0}b(3n+1)q^n=\frac{q^{-1}}{2^6\psi(q)^6}\left(2\varphi(-q)^2+7\frac{\varphi(-q^3)^4}{\varphi(-q)^2}-36\frac{\varphi(-q^3)^8}{\varphi(-q)^6}+27\frac{\varphi(-q^3)^{12}}{\varphi(-q)^{10}}\right).$$

Using the same method, we can obtain

\begin{align*}
&\sum_{n\ge 0}b(9n+7)q^n\\
&\quad=\frac{q^{-3}}{2^{24}\psi(q)^{18}}\bigg(252 \frac{\varphi(-q)^{18}}{\varphi(-q^3)^4}+16254 \varphi(-q)^{14} + 54054 \varphi(-q)^{10}\varphi(-q^3)^{4} \\
&\quad\quad+ 54180 \varphi(-q)^{6}\varphi(-q^3)^{8} - 3679992 \varphi(-q)^{2}\varphi(-q^3)^{12}+ 33485805 \frac{\varphi(-q^3)^{16}}{\varphi(-q)^{2}} \\
&\quad\quad - 201778452\frac{\varphi(-q^3)^{20}}{\varphi(-q)^{6}} + 846955116 \frac{\varphi(-q^3)^{24}}{\varphi(-q)^{10}} - 2445337188 \frac{\varphi(-q^3)^{28}}{\varphi(-q)^{14}}\\
&\quad\quad+ 4746831012 \frac{\varphi(-q^3)^{32}}{\varphi(-q)^{18}} - 6004220418 \frac{\varphi(-q^3)^{36}}{\varphi(-q)^{22}} + 4706441496 \frac{\varphi(-q^3)^{40}}{\varphi(-q)^{26}}\\
&\quad\quad- 2066242608 \frac{\varphi(-q^3)^{44}}{\varphi(-q)^{30}} + 387420489 \frac{\varphi(-q^3)^{48}}{\varphi(-q)^{34}}\bigg).
\end{align*}

Checking the coefficients $252,16254,\ldots$, we can see that they are all divisible by $9$. It therefore follows that
$$b(9n+7)\equiv 0\pmod{9}$$
for all $n\ge 0$, which was obtained previously by Zhao and Zhong \cite{ZZ2011}.

For the following dissections, the coefficients become very large. Hence we only display them under modulus $2187=3^7$. We have
\begin{align*}
&\sum_{n\ge 0}b(27n+7)q^n\\
&\quad\equiv\frac{q^{-7}}{2^{78}\psi(q)^{54}}\bigg(252 \frac{\varphi(-q)^{74}}{\varphi(-q^3)^{24}} + 504 \frac{\varphi(-q)^{70}}{\varphi(-q^3)^{20}} +918 \frac{\varphi(-q)^{66}}{\varphi(-q^3)^{16}} + 2034 \frac{\varphi(-q)^{62}}{\varphi(-q^3)^{12}}\\
&\quad\quad  + 396 \frac{\varphi(-q)^{58}}{\varphi(-q^3)^{8}} + 1755 \frac{\varphi(-q)^{54}}{\varphi(-q^3)^{4}} + 225 \varphi(-q)^{50} + 1530 \varphi(-q)^{46}\varphi(-q^3)^{4}\\
&\quad\quad  + 1701 \varphi(-q)^{42}\varphi(-q^3)^{8} + 1620 \varphi(-q)^{38}\varphi(-q^3)^{12}\bigg) \pmod{2187}.
\end{align*}

Continuing the dissection, we have
\begin{align*}
&\sum_{n\ge 0}b(81n+61)q^n\\
&\quad\equiv\frac{q^{-21}}{2^{240}\psi(q)^{162}}\bigg(729 \frac{\varphi(-q)^{230}}{\varphi(-q^3)^{72}} + 1458 \frac{\varphi(-q)^{226}}{\varphi(-q^3)^{68}} + 1458 \frac{\varphi(-q)^{218}}{\varphi(-q^3)^{60}} + 729 \frac{\varphi(-q)^{214}}{\varphi(-q^3)^{56}}\\
&\quad\quad + 729 \frac{\varphi(-q)^{194}}{\varphi(-q^3)^{36}} + 1458 \frac{\varphi(-q)^{190}}{\varphi(-q^3)^{32}} + 1458 \frac{\varphi(-q)^{182}}{\varphi(-q^3)^{24}} + 729 \frac{\varphi(-q)^{178}}{\varphi(-q^3)^{20}}\\
&\quad\quad + 729 \varphi(-q)^{158} + 1458 \varphi(-q)^{154}\varphi(-q^3)^{4} + 1458 \varphi(-q)^{146}\varphi(-q^3)^{12}\\
&\quad\quad + 729 \varphi(-q)^{142}\varphi(-q^3)^{16}\bigg) \pmod{2187}.
\end{align*}
Now by checking the coefficients modulo $729$, we can see that
$$b(81n+61)\equiv 0\pmod{729},$$
and thus complete the proof of Theorem \ref{th:1.1}.

Furthermore, we can rewrite $\sum_{n\ge 0}b(81n+61)q^n$ as
\begin{align*}
&\sum_{n\ge 0}b(81n+61)q^n\\
&\quad\equiv\frac{729q^{-21}}{2^{240}\psi(q)^{162}}\varphi(-q)^{142}\varphi(-q^3)^{16}\\
&\quad\quad\quad\times(1+\kappa(q)^{9}+\kappa(q)^{18})(1+2\kappa(q)+2\kappa(q)^3+\kappa(q)^4) \pmod{2187},
\end{align*}
where
$$\kappa(q)=\frac{\varphi(-q)^4}{\varphi(-q^3)^4}.$$
According to \eqref{eq:2.5}, we have
$$\kappa(q)=1-8qw(q)^3,$$
and thus
$$\kappa(q)\equiv 1-2qw(q^3)=1-\xi(q)\pmod{3}.$$
To prove Theorem \ref{th:1.1.5}, we note that
\begin{align*}
&\frac{q^{-21}}{2^{240}\psi(q)^{162}}\varphi(-q)^{142}\varphi(-q^3)^{16}(1+\kappa(q)^{9}+\kappa(q)^{18})(1+2\kappa(q)+2\kappa(q)^3+\kappa(q)^4)\\
&\quad\equiv\frac{(q^{3})^{-7}}{\psi(q^3)^{54}}(\varphi(-q)\varphi(-q^3)^{47})\varphi(-q^3)^{16}(1+\kappa(q^3)^{3}+\kappa(q^3)^{6})(1-\kappa(q))^4\\
&\quad\equiv \frac{(q^{3})^{-7}}{\psi(q^3)^{54}}\varphi(-q^3)^{63}\varphi(-q^9)(1+\kappa(q^3)^{3}+\kappa(q^3)^{6})(1-\xi(q))\xi(q)^4 \pmod{3}.
\end{align*}
Hence
\begin{align*}
&\sum_{n\ge 0}b(81n+61)q^n\\
&\quad\equiv\frac{729(q^{3})^{-7}}{\psi(q^3)^{54}}\varphi(-q^3)^{63}\varphi(-q^9)(1+\kappa(q^3)^{3}+\kappa(q^3)^{6})(1-\xi(q))\xi(q)^4 \pmod{2187}.
\end{align*}
Since $\xi(q)=2qw(q^3)$, we can see that the right hand side contains no term of the form $q^{3n}$. Hence
$$b(243n+61)\equiv 0 \pmod{2187}$$
holds for all $n\ge 0$.

We next prove Theorem \ref{th:1.2}. If we extract terms of the form $q^{3n}$ from $\sum_{n\ge 0}b(27n+7)q^n$ and replace $q^3$ by $q$, we have
\begin{align*}
&\sum_{n\ge 0}b(81n+7)q^n\\
&\quad\equiv\frac{9q^{-20}w(q)\varphi(-q)^{154}\varphi(-q^3)^{4}}{2^{239}\psi(q)^{162}}\\
&\quad\quad\quad\times(5+ 8\kappa(q)+ 5\kappa(q)^{2}+ 6\kappa(q)^{3}+ 6\kappa(q)^{4}+ 6\kappa(q)^{5}+ 3\kappa(q)^{6}\\
&\quad\quad\quad\quad\quad+ 3\kappa(q)^{7}+ 3\kappa(q)^{8}+ 8\kappa(q)^{9}+ 2\kappa(q)^{10}+ 8\kappa(q)^{11}+ 3\kappa(q)^{12}\\
&\quad\quad\quad\quad\quad+ 3\kappa(q)^{13}+ 3\kappa(q)^{14}+ 6\kappa(q)^{15}+ 6\kappa(q)^{16}+ 6\kappa(q)^{17}+ 5\kappa(q)^{18}\\
&\quad\quad\quad\quad\quad+ 8\kappa(q)^{19}+ 5\kappa(q)^{20}) \pmod{81}.
\end{align*}
Note that
\begin{align*}
(1-x)^{20}&\equiv1+ 7x+ x^{2}+ 3x^{3}+ 3x^{4}+ 3x^{5}+ 6x^{6}+ 6x^{7}+ 6x^{8}\\
&\quad\quad+ 7x^{9}+ 4x^{10}+ 7x^{11}+ 6x^{12}+ 6x^{13}+ 6x^{14}+ 3x^{15}+ 3x^{16}\\
&\quad\quad+ 3x^{17}+ x^{18}+ 7x^{19}+ x^{20}\pmod{9}.
\end{align*}
We therefore have
\begin{align*}
&\sum_{n\ge 0}b(81n+7)q^n\\
&\quad\equiv9\times\frac{5q^{-20}w(q)\varphi(-q)^{154}\varphi(-q^3)^{4}}{2^{239}\psi(q)^{162}}(1-\kappa(q))^{20}\\
&\quad\equiv9\frac{(q;q)_\infty^{531}(q^6;q^6)_\infty^{179}}{(q^2;q^2)_\infty^{539}(q^3;q^3)_\infty^{175}}\pmod{81}.\\
\end{align*}
One readily sees from the binomial theorem that
$$(q;q)_\infty^9\equiv(q^3;q^3)_\infty^3\pmod{9}.$$
We therefore conclude that
\begin{align*}
&\sum_{n\ge 0}b(81n+7)q^n\\
&\quad\equiv9(q^3;q^3)_\infty^2\left(\frac{(q;q)_\infty^{9}}{(q^3;q^3)_\infty^{3}}\right)^{59} \frac{(q^2;q^2)_\infty}{(q^6;q^6)_\infty} \left(\frac{(q^6;q^6)_\infty^{3}}{(q^2;q^2)_\infty^{9}}\right)^{60}\\
&\quad\equiv 9\frac{(q^2;q^2)_\infty(q^3;q^3)_\infty^2}{(q^6;q^6)_\infty}\pmod{81}.
\end{align*}
This proves \eqref{eq:lin2}.

At last, we extract terms of the form $q^{3n+1}$ from $\sum_{n\ge 0}b(27n+7)q^n$ and replace $q^3$ by $q$. Then
\begin{align*}
&\sum_{n\ge 0}b(81n+34)q^n\\
&\quad\equiv9\frac{q^{-20}w(q)^2\varphi(-q)^{154}\varphi(-q^3)^{4}}{2^{238}\psi(q)^{162}}\\
&\quad\quad\quad\times(1+ 7\kappa(q)+ \kappa(q)^{2}+ 3\kappa(q)^{3}+ 3\kappa(q)^{4}+ 3\kappa(q)^{5}+ 6\kappa(q)^{6}\\
&\quad\quad\quad\quad\quad+ 6\kappa(q)^{7}+ 6\kappa(q)^{8}+ 7\kappa(q)^{9}+ 4\kappa(q)^{10}+ 7\kappa(q)^{11}+ 6\kappa(q)^{12}\\
&\quad\quad\quad\quad\quad+ 6\kappa(q)^{13}+ 6\kappa(q)^{14}+ 3\kappa(q)^{15}+ 3\kappa(q)^{16}+ 3\kappa(q)^{17}+ \kappa(q)^{18}\\
&\quad\quad\quad\quad\quad+ 7\kappa(q)^{19}+ \kappa(q)^{20})\\
&\quad\equiv 9\frac{q^{-20}w(q)^2\varphi(-q)^{154}\varphi(-q^3)^{4}}{2^{238}\psi(q)^{162}}(1-\kappa(q))^{20}\\
&\quad\equiv 36\frac{(q;q)_\infty^{532}(q^6;q^6)_\infty^{182}}{(q^2;q^2)_\infty^{540}(q^3;q^3)_\infty^{178}}\\
&\quad\equiv 36\frac{(q;q)_\infty}{(q^3;q^3)_\infty}\left(\frac{(q;q)_\infty^{9}}{(q^3;q^3)_\infty^{3}}\right)^{59}(q^6;q^6)_\infty^2\left(\frac{(q^6;q^6)_\infty^{3}}{(q^2;q^2)_\infty^{9}}\right)^{60}\\
&\quad\equiv 36\frac{(q;q)_\infty(q^6;q^6)_\infty^2}{(q^3;q^3)_\infty}\pmod{81}.
\end{align*}
We therefore end the proof of \eqref{eq:lin3}.

\section{Concluding remarks}

Recall that from Lemma \ref{le:2.4}, for any positive integer $\alpha$, we can find some $\ell$, such that $\sum_{n\ge 0}b(3^\alpha n+\ell)q^n$ has the form
$$\frac{q^\alpha}{\psi(q)^\beta}\sum_{i=1}^n v(s_i,t_i)\varphi(-q)^{s_i}\varphi(-q^3)^{t_i},$$
where $s_1\equiv s_2\equiv\cdots\equiv s_n\equiv 2 \pmod{4}$, $t_1\equiv t_2\equiv\cdots\equiv t_n\equiv 0 \pmod{4}$, and $s_1+t_1=s_2+t_2=\cdots=s_n+t_n$. In fact, it is easy to show by induction that
$$\ell=
\begin{cases}
1+\frac{3^{\alpha+1}-3}{4} & \alpha \text{ even},\\
1+\frac{3^{\alpha}-3}{4} & \alpha \text{ odd}.
\end{cases}$$

From Theorem \ref{th:1.1.5}, we see that for $\alpha=6$ and $7$ (and thus $\ell=547$ in both cases),
$$b(729n+547)\equiv 0\pmod{2187}$$
and
$$b(2187n+547)\equiv 0\pmod{2187}.$$
However the modulus $2187=3^7$ is the best choice since
\begin{align*}
b(547)&=2135474526556068875092854278074796547960\\
&=2^3\times 3^7\times 5\times 41\times 61\times 151\times 11909\times 5427748132276664632973303.
\end{align*}
We therefore want to know if there is a general family of congruences modulo higher power of $3$ for $b(3^\alpha n+\ell)$?

\section*{Acknowledgement}

Many thanks go to George E. Andrews for encouraging me in the study of partition theory.

\bibliographystyle{amsplain}

\end{document}